\documentclass[a4paper,12pt,reqno]{amsart}
\usepackage{amsfonts, color}
\usepackage{amsmath}
\usepackage{amssymb}
\usepackage[a4paper]{geometry}
\usepackage{mathrsfs}
\usepackage[colorlinks]{hyperref}
\usepackage{comment}
\usepackage[colorinlistoftodos]{todonotes}

\usepackage{hyperref}
\renewcommand\eqref[1]{(\ref{#1})}

\makeatletter
\newcommand*{\mint}[1]{%
  \mint@l{#1}{}%
}
\newcommand*{\mint@l}[2]{%
  \@ifnextchar\limits{%
    \mint@l{#1}%
  }{%
    \@ifnextchar\nolimits{%
      \mint@l{#1}%
    }{%
      \@ifnextchar\displaylimits{%
        \mint@l{#1}%
      }{%
        \mint@s{#2}{#1}%
      }%
    }%
  }%
}
\newcommand*{\mint@s}[2]{%
  \@ifnextchar_{%
    \mint@sub{#1}{#2}%
  }{%
    \@ifnextchar^{%
      \mint@sup{#1}{#2}%
    }{%
      \mint@{#1}{#2}{}{}%
    }%
  }%
}
\def\mint@sub#1#2_#3{%
  \@ifnextchar^{%
    \mint@sub@sup{#1}{#2}{#3}%
  }{%
    \mint@{#1}{#2}{#3}{}%
  }%
}
\def\mint@sup#1#2^#3{%
  \@ifnextchar_{%
    \mint@sup@sub{#1}{#2}{#3}%
  }{%
    \mint@{#1}{#2}{}{#3}%
  }%
}
\def\mint@sub@sup#1#2#3^#4{%
  \mint@{#1}{#2}{#3}{#4}%
}
\def\mint@sup@sub#1#2#3_#4{%
  \mint@{#1}{#2}{#4}{#3}%
}
\newcommand*{\mint@}[4]{%
  \mathop{}%
  \mkern-\thinmuskip
  \mathchoice{%
    \mint@@{#1}{#2}{#3}{#4}%
        \displaystyle\textstyle\scriptstyle
  }{%
    \mint@@{#1}{#2}{#3}{#4}%
        \textstyle\scriptstyle\scriptstyle
  }{%
    \mint@@{#1}{#2}{#3}{#4}%
        \scriptstyle\scriptscriptstyle\scriptscriptstyle
  }{%
    \mint@@{#1}{#2}{#3}{#4}%
        \scriptscriptstyle\scriptscriptstyle\scriptscriptstyle
  }%
  \mkern-\thinmuskip
  \int#1%
  \ifx\\#3\\\else_{#3}\fi
  \ifx\\#4\\\else^{#4}\fi
}
\newcommand*{\mint@@}[7]{%
  \begingroup
    \sbox0{$#5\int\m@th$}%
    \sbox2{$#5\int_{}\m@th$}%
    \dimen2=\wd0 %
    \let\mint@limits=#1\relax
    \ifx\mint@limits\relax
      \sbox4{$#5\int_{\kern1sp}^{\kern1sp}\m@th$}%
      \ifdim\wd4>\wd2 %
        \let\mint@limits=\nolimits
      \else
        \let\mint@limits=\limits
      \fi
    \fi
    \ifx\mint@limits\displaylimits
      \ifx#5\displaystyle
        \let\mint@limits=\limits
      \fi
    \fi
    \ifx\mint@limits\limits
      \sbox0{$#7#3\m@th$}%
      \sbox2{$#7#4\m@th$}%
      \ifdim\wd0>\dimen2 %
        \dimen2=\wd0 %
      \fi
      \ifdim\wd2>\dimen2 %
        \dimen2=\wd2 %
      \fi
    \fi
    \rlap{%
      $#5%
        \vcenter{%
          \hbox to\dimen2{%
            \hss
            $#6{#2}\m@th$%
            \hss
          }%
        }%
      $%
    }%
  \endgroup
}
\numberwithin{equation}{section}
\theoremstyle{plain}
\newtheorem{thm}{Theorem}[section]

\newtheorem{cor}[thm]{Corollary}
\newtheorem{lem}[thm]{Lemma}

\theoremstyle{definition}
\newtheorem{defn}[thm]{Definition}
\newtheorem{rem}[thm]{Remark}

\newtheorem{cons}[thm]{Consequence}
\newcommand{\G}{\mathbb{G}}

\def\G{\mathbb{G}}

\def\H{\mathbb{H}^{n}}
\def\Lh{\mathcal{L}}
\def\Hh{\mathbb{H}^{1}}

\def\L{\mathcal{L}_{p}}

\def\V{V(\sqrt{t})}
\def\G{{\mathbb G}}
\def\L{\mathcal{L}_{\mathbb{G}}}
\def\LL{\mathcal{L}_{\mathbb{H}^1}}
\def\LLL{\mathcal{L}_{\mathbb{H}^{n}}}

\def\H{\mathbb{H}^{n}}

\def\x{{\xi}}
\def\t{{\tau}}

\title[On global solutions of time-dependent heat equations]{On  global solutions of heat equations with time-dependent nonlinearities on unimodular Lie groups}
\author[M. Chatzakou]{Marianna Chatzakou}
\address{
	Marianna Chatzakou:
	\endgraf
    Department of Mathematics: Analysis, Logic and Discrete Mathematics
    \endgraf
    Ghent University, Belgium
  	\endgraf
	{\it E-mail address} {\rm marianna.chatzakou@ugent.be}
		}
	
\author[A. Kassymov]{Aidyn Kassymov}
\address{
  Aidyn Kassymov:
  \endgraf
  Institute of Mathematics and Mathematical Modeling
  \endgraf
  28 Shevchenko str.
  \endgraf
  050010 Almaty
  \endgraf
  Kazakhstan
  \endgraf
	{\it E-mail address}  {\rm kassymov@math.kz}}

\author[M. Ruzhansky]{Michael Ruzhansky}
\address{
  Michael Ruzhansky:
  \endgraf
  Department of Mathematics: Analysis, Logic and Discrete Mathematics
  \endgraf
  Ghent University, Belgium
  \endgraf
 and
  \endgraf
  School of Mathematical Sciences
  \endgraf
  Queen Mary University of London
  \endgraf
  United Kingdom
  \endgraf
  {\it E-mail address} {\rm michael.ruzhansky@ugent.be}
  }

\begin{document}

\thanks{The authors are supported by the FWO Odysseus 1 grant G.0H94.18N: Analysis and Partial
Differential Equations and by the Methusalem programme of the Ghent University Special Re-
search Fund (BOF) (Grant number 01M01021). Marianna Chatzakou is a postdoctoral fellow of
the Research Foundation – Flanders (FWO) under the postdoctoral grant No 12B1223N. Michael
Ruzhansky is also supported by the EPSRC grant  EP/V005529/1 and FWO grant G022821N. Aidyn
Kassymov is also supported by the Science Committee of the Ministry of Science and Higher Education of the Republic of Kazakhstan (Grant No. AP23484106). \\
\indent
{\it Keywords:} Nonlinear heat equation; unimodular Lie group; global well-posedeness; sub-Laplacian.}

\begin{abstract} 
In this work, we study the global well-posedeness of the heat equation with variable time-dependent nonlinearity of the form $\varphi(t)f(u)$ on unimodular Lie groups when the differential operator arises as the sum of squares of H\"ormander vector fields. For general unimodular Lie groups, we derive the necessary conditions for the nonexistence of global positive solutions. This gives different conditions in the cases of compact, polynomial, and exponential volume growth groups. In the case of the Heisenberg groups $\H$, we also derive sufficient conditions, which coincide with the necessary ones in the case of $\Hh$ (and this is also true for $\mathbb{R}^{n}$).
 In particular, in the case of the Heisenberg group $\Hh$ we obtain the necessary and sufficient conditions under which the aforesaid initial value problem with variable nonlinearity has a global positive solution. 
\end{abstract}

\maketitle

\section{Introduction}
In a pioneering work \cite{F66}, Fujita  proved that the initial heat problem
\begin{equation}\label{Heat_Cauchy}
\begin{split}
& u_t=\Delta u+u^p,\,\, x\in \mathbb{R}^N,\, t>0, \\&
u(0,x)=u_0(x)\geq 0,\, x\in \mathbb{R}^N,
   \end{split}
\end{equation}
has no global in-time non-negative solution for
\begin{equation}\label{CrExFuj}
1<p<p_F=1+\frac{2}{N},
\end{equation}
with $u_0\not\equiv 0$. In the same direction the authors of \cite{H73} considered the problem \eqref{Heat_Cauchy} in Euclidean spaces with $N=1,2$ for $p=p_{F}$, while in \cite{KST, W81}, the latter case was studied for general $N$, where non-existence was also proved. In \cite{W81} the proof is rather simple, compared with the existing methods, and additionally, the author proved that under certain conditions the  solutions to \eqref{Heat_Cauchy} when $N(p-1) \leq 1$ blow up in the $L^p$-norm. In the literature, the critical exponent $p_{F}$ plays a significant role in investigating the global existence and nonexistence for the Cauchy problem \eqref{Heat_Cauchy}, and it is known as the Fujita critical exponent, named after Fujita.

Variations of the problem \eqref{Heat_Cauchy} were considered that include cases where $u^p$ is replaced by some function $f$ applied on the solution $u$, or even cases where a time-dependent non-linearity, say of the form $h(t)$, appears as an extra term in front of $f(u)$. Let us also mention  the case where the differential operator is replaced by the $p$-Laplacian, which has been studied e.g. in \cite{CH21}. In \cite{M90}, the author studied a variation of the initial bound problem \eqref{Heat_Cauchy} with the term $h(t)\cdot u^p$ on bounded and unbounded domains of $\mathbb{R}^N$ with different behaviour of $h$ in each case. In \cite{F68} Fujita also studied the solutions to the initial value problem  of the form
\begin{equation*}
    \begin{split}
        &\ u_t=\Delta u+f(u)\,, x \in \Omega \subset \mathbb{R}^N,\\&
   u|_{t=0}=u_{0}(x)\,, u|_{\partial \Omega}=0\,,     
    \end{split}
\end{equation*}
for different choices of $f$ (see also \cite{CH24}). More generally, on a sub-Riemannian manifold $M$,  the well-posedness of the Cauchy problem of the form 
\begin{equation}
\label{RY}
    \begin{split}
        &\ u_t=\mathcal{L}_{M} u+f(u)\,, x \in M\,, \\&
   u(0,x)=u_0(x)\,, x \in M\,,     
    \end{split}
\end{equation}
for $u_0 \geq 0$, was studied in \cite{RY22}, where $\mathcal{L}_{M}$ stands for the generalisation of the Laplace--Beltrami operator in the sub-Riemannian setting $M$. This included the determination of the Fujita exponent for general unimodular Lie groups. Also, in \cite{P98} the author investigated the Fujita exponent on nilpotent groups in the case $f(u)=u^{p}$. 

Here we generalise the result in \cite{RY22} for unimodular Lie groups by adding a time-dependent function $\varphi$ in front of $f(u)$. In particular, we study the global well-posedness of the nonlinear heat equation on a unimodular Lie group $\G$ of the form:
\begin{equation}\label{1}
    u_{t}(t,x)=\L u(t,x)+\varphi(t)f(u(t,x)),\,\,\,\,\,(t,x)\in \mathbb{R}_{+}\times \G,
\end{equation}
where $\varphi$ is a non-negative function, and $f$ is a continuous and non-negative function on $[0, \infty)$ with $f(0) = 0$ and $f(v) > 0$ for $v > 0$, with initial data 
\begin{equation}\label{2}
    u(0,x)=u_{0}(x),\,\,\,\,\,x\in \G.
\end{equation}
Specifically, our first result, Theorem \ref{thm1}, is general and provides a sufficient condition under which the initial value problem \eqref{1}-\eqref{2} has a global solution when $\G$ is unimodular Lie group. In Section \ref{SEC:4} we study the converse, deriving the necessary conditions for the existence of global solutions in the case of Heisenberg groups $\H$. In Corollary \ref{cor1} we see that the necessary and sufficient conditions coincide in case of the Heisenberg group $\Hh$ of topological dimension $3$.

Non-linear diffusion equations naturally appear in other sciences, and the existing literature on the topic is vast.  For instance, we refer to  \cite{AW78} where the authors studied the equation \eqref{Heat_Cauchy} in the context of population genetics where $u^p$ is replaced by $f(u)$ for some suitable for their purposes function $f$. For other works on the Fujita exponent where the non-mathematical motivation is discussed, we refer to \cite{BLMW02, A17, M13} to mention only very few.

\section{Preliminaries}
Let $(M, \mu)$ be a manifold with a measure $\mu$, and let $X_1,\ldots, X_k$ be vector fields on $M$ that are formally skew-adjoint on $L^2(M, \mu)$; i.e., we have 
\[
\int_{M} (X_i f)gd\mu=-\int_{M}f (X_ig)d\mu\,,\quad f,g \in C_{0}^{\infty}(M)\,.
\]
We will be mostly concerned with the case $M=\G$ being  a unimodular Lie group; in which case $d\mu$ is the bi-invariant Haar measure on $\G$. Then, the operator $e^{t \L}$ that arises from the Friedrichs extension of the sum of squares $\L=\sum\limits_{i=1}^{k}X_{i}^{2}$ of the vector fields $X_1,\ldots, X_k$ is a symmetric semigroup of contractions on $L^2(M, \mu)$. 

In this work, we consider a family $X_1,\ldots, X_k$ of $C^{\infty}$ vector fields on a manifold $M$ such that, at every point of $M$, the Lie algebra they generate is the whole tangent space of $M$; such vector fields are said to satisfy H\"ormander's condition. Roughly speaking, the diffusion process generated by the operator $\L=\sum\limits_{i=1}^{k} X_{i}^{2}$ ``follows'' the curves determined by the vector fields $X_1,\ldots, X_k$, and the path between any two points in $M$ stays tangent to the span of $X_1,\ldots, X_k$. This way one can define a distance, called the Carnot-Carath\'eodory distance, which is suitable for the study of the operator $\L$.

A special class of sub-Riemannian manifolds where the latter situation appears naturally is that of nilpotent Lie groups. Let us recall their definition: we say that the connected Lie groups $\G$ with Lie algebra $\mathfrak{g}$ is nilpotent (of rank $r$), if $\mathfrak{g}_{r+1}=\{0\}$, where $\mathfrak{g}_1=\mathfrak{g}$, and we have defined $\mathfrak{g}_i=[\mathfrak{g}, \mathfrak{g}_{i-1}]$.  Recall the $\mathfrak{g}$ can be identified with the left-invariant vector fields on $\G$.

In the sequel, we denote by $\G$ a unimodular Lie group equipped with the natural bi-invariant Haar measure, and by $\textbf{X}=\{X_1,\ldots, X_k \}$ a system of (left-invariant) H\"ormander vector fields on $\G$. If we denote by $\L=\sum\limits_{i=1}^{k}X_{i}^{2}$ the negative hypoelliptic operator on $\G$, with associated Markovian semigroup $e^{t\L}$, $t>0$, then the latter is given by the right convolution 
\begin{equation}\label{def.heat}
    e^{t\L}g(x):=\int_{\G}p_{t}(y^{-1}x)g(y)dy\,,
\end{equation}
where the so-called heat kernel $p_t$ is a positive solution of $\left(\frac{\partial}{\partial t}-\L \right)u=0$ with $\delta$  as the initial condition. 

The next theorem gives us upper and lower estimates for $p_t$ in the case where the latter arises by the semigroup $e^{t\L}$ as above. Its behaviour depends on how the ball with a radius of $t$ grows, specifically, whether it exhibits exponential or polynomial types.

As in \cite{VCSC92}, in the polynomial volume growth case, we say that the volume growth of $\G$ follows $V(r)\simeq r^{D}$ for $r\geq1$ and $V(r)\simeq r^{d}$ for $0\leq r\leq1$, where $D$ and $d$ are global and local dimensions, respectively, and $V(r)$ is the measure of the ball with radius $r$. 

We note that the global dimension $D$ depends only on the group $\G$ but not on the system of H\"ormander's vector fields, see e.g. \cite[Page 285]{CRTN01}.

We recall  the heat kernel estimate on unimodular Lie group in the following form:
\begin{thm}[Theorems VIII.2.9, VIII.4.3 \cite{VCSC92}]
    Let $\G$ be a unimodular Lie group and $\textbf{X}$ be a H\"ormander system of left-invariant vector fields on $\G$. If $\G$ has a polynomial volume growth, then  there exists $C,c>0$, such that 
    \begin{equation}
        \label{asymptpoly}
        c \V^{-1}e^{-\frac{C\rho^{2}(x)}{t}}\leq p_t(x) \leq C \V^{-1}e^{-\frac{c\rho^2(x)}{t}}\,, 
    \end{equation}
    for all $x \in \G$, and $t>0$, where $\rho$ is the Carnot-Carath\'eodory distance associated with $\textbf{X}$.

    If $\G$ has exponential growth and if the local dimension of $(\G,\textbf{X})$ is $d$, then for any $\nu\geq d$ and $\varepsilon>0$, there exists $C_{\nu,\varepsilon}>0$,  such that
    \begin{equation}
        \label{asymptexp}
        |p_t(x)| \leq C_{\nu,\epsilon} t^{-\frac{\nu}{2}}e^{-\frac{\rho^2(x)}{(4+\epsilon)t}}\,, 
    \end{equation}
    for all $x \in \G$ and $t>0$.
\end{thm}

\begin{rem}[Theorems IV.4.2, IV.4.3, IV.5.8 \cite{VCSC92} ]
    Let $\G$ be a nilpotent  Lie group, and let $\textbf{X}$ be a H\"ormander system of left-invariant vector fields on $\G$. Then, there exists $\epsilon>0$, and constants $C,C'>0$ and $C_\epsilon>0$, depending on $\epsilon$, such that 
    \begin{equation}
        \label{asympt}
        C \V^{-1}e^{-\frac{C'\rho^{2}(x)}{t}}\leq p_t(x) \leq C_\epsilon \V^{-1}e^{-\frac{\rho^2(x)}{4(1+\epsilon)t}}\,, 
    \end{equation}
    for all $x \in \G$, and $t>0$, where $\rho$ is the Carnot-Carath\'eodory distance associated with $\textbf{X}$, and $V(t)$ stands for the volume of the ball $B(t)$ with respect to $\rho$ of a fixed center and of radius $t$. In the case of a homogeneous Lie group $\G$, we have $D=d=Q$, the homogeneous dimension of the homogeneous Lie group $\G$. 
\end{rem}
It is easy to check that the solution to the Cauchy problem  \eqref{1}-\eqref{2} is given by
\begin{equation}\label{expr.u0}
\begin{split}
      u(t,x)&=e^{t\L}u_{0}(x)+\int_{0}^{t}e^{(t-\tau)\L}\varphi(\tau)f(u(\tau,x))d\tau.
\end{split}
\end{equation}

As usual, we say that the (positive) solution  to the heat equation \eqref{1} with the initial data \eqref{2} exists globally if  $\|u(t,\cdot)\|_{L^{\infty}(\G)}<\infty$ for all $t>0$. If for some $t^{*}>0$ we have $\|u(t,\cdot)\|_{L^{\infty}(\G)}\rightarrow\infty$ when $t\rightarrow t^{*}$, then we say that the solution blows up in finite time.

Let us give definition of  minorant and majorant functions:
\begin{defn}
    Let us set functions $f_{m}:[0,+\infty)\rightarrow [0,+\infty) $ and $f_{M}:[0,+\infty)\rightarrow [0,+\infty)$ as minorant and majorant functions defined by:
    \begin{equation}
        f_{m}(v):=\inf_{\alpha\in(0,1)}\frac{f(\alpha v)}{f(\alpha)},\,\,\,\,v\geq0,
    \end{equation}
    and 
     \begin{equation}
        f_{M}(v):=\sup_{\alpha\in(0,1)}\frac{f(\alpha v)}{f(\alpha)},\,\,\,\,v\geq0,
    \end{equation}
   respectively.
\end{defn}
From this definition, we have 
\begin{equation}\label{maj1}
f(\alpha)f_{m}(v)\leq f(\alpha v)\leq f(\alpha)f_{M}(v),\,\,\,\alpha\in (0,1),\,\,\,\,v\geq 0.    
\end{equation}
 For the function $f$ as in \eqref{1} we will sometimes assume the following conditions on the majorant and minorant functions associated to it:
\begin{equation}
    \label{maj}
    \lim_{v \rightarrow 0^{+}}\frac{f_M(v)}{v}=0\,,
\end{equation}
and 
\begin{equation}
    \label{min}
   \int_{1}^{+\infty}\frac{1}{f_{m}(v)}dv<+\infty\,.
\end{equation}
The reason for the above requirements will become apparent  in the proofs of Theorems \ref{thm1} and \ref{onhn}.

Also, by $C_{0}(\G)$ we denote the space of continuous functions vanishing at infinity.

\section{Main result}
Our first result applies to any unimodular Lie group $\G$ and provides the necessary conditions under which the initial value problem \eqref{1}-\eqref{2} has no global solutions or, alternatively, a sufficient condition for the existence of positive global solutions.
\begin{thm}\label{thm1}
     Let $\G$ be a unimodular Lie group and let $\textbf{X}$ be a H\"ormander system of left-invariant vector fields on $\G$. Assume that $f$ is a  continuous, non-negative function with $f(0)=0$ and $f(v)>0$ when $v>0$, satisfying $\eqref{maj}$, the mapping  $v\mapsto\frac{f(v)}{v}$ is non-decreasing, and $0\leq\varphi\in L^{1}_{\text{loc}}[0,+\infty)$. 
     \begin{itemize}
         \item[(i)] Let $\G$ have a polynomial growth. If \eqref{1}-\eqref{2} does not have a global solution for any $0<u_{0}\in C_{0}(\G)\cap L^{1}(\G)$, then for all $0<w\in C_{0} (\G)\cap L^{1}(\G)$ we have 
\begin{equation}\label{unimodcond1}
\int\limits_{0}^{\infty}\varphi(\tau)\frac{f(\|e^{\tau\L}w\|_{L^{\infty}(\G)})}{\|e^{\tau\L}w\|_{L^{\infty}(\G)}}d\tau=+\infty.    
\end{equation}

Furthermore, in this case, for every $\theta>0$, we also have 
\begin{equation}\label{cond2poly}
    \int_{1}^{\infty}\varphi(\tau)\tau^{\frac{D}{2}} f(\theta \tau^{-\frac{D}{2}})d\tau=+\infty,
\end{equation}
where $D$ is the global dimension of $\G$.

\item[(ii)] Let $\G$ have an exponential growth. If \eqref{1}-\eqref{2} does not have a global solution for any $0<u_{0}\in C_{0}(\G)\cap L^{1}(\G)$, then for all $0<w\in C_{0} (\G)\cap L^{1}(\G)$ we have \eqref{unimodcond1}.

Moreover, in this case, for every $\theta>0$, we also have 
\begin{equation}\label{cond2exp}
    \int_{1}^{\infty}\varphi(\tau)\tau^{\frac{\nu}{2}} f(\theta \tau^{-\frac{\nu}{2}})d\tau=+\infty,
\end{equation}
for any $\nu\geq d$.
 \end{itemize}

\end{thm}

\begin{rem}
    Theorem \ref{thm1} gives a criterion for the existence of global positive solutions for \eqref{1}-\eqref{2}. Namely, if $\G$ has polynomial growth, and $f$ and $\varphi$ are such that 
    $$\int_{1}^{\infty}\varphi(\tau)\tau^{\frac{D}{2}}f(\theta\tau^{-\frac{D}{2}})<\infty,\,\,\,\text{for\,\,\,some} \,\,\theta>0,$$
    then there exists $0<u_{0}\in C_{0}(\G)\cap L^{1}(\G)$ such that \eqref{1}-\eqref{2} has a positive global solution. Similarly, if $\G$ has exponential growth and there exists $\theta>0$ and $\nu\geq d$ such that the integral in \eqref{cond2exp} is finite, then there exists $0<u_{0}\in C_{0}(\G)\cap L^{1}(\G)$ such that \eqref{1}-\eqref{2} has a positive global solution.
\end{rem}
Before proceeding to proving Theorem \ref{thm1}, let us give several of its consequences in special cases of different types of groups and nonlinearities. Especially, in the special case $\varphi(t)\equiv 1$, we recover some results of \cite{RY22}.

We also note that in the commutative case $\G=\mathbb R^d$,  the result of Theorem \ref{thm1} with the condition \eqref{cond2poly} with $D=d$, implies (in the sense that our assumptions on $f$ are slightly less restrictive) the corresponding necessity result in \cite{CH22}.

\begin{cor}[Groups of polynomial growth]
    Let $\G$ have a polynomial volume growth, and let $\varphi(t)\equiv 1$ and $f(v)=v^{p}$ with $p>1$. If $\eqref{1}$-$\eqref{2}$ does not have global positive solutions, then we have  $1<p\leq p_{c}= 1+\frac{2}{D}$. Alternatively, if $p>p_c$ then $\eqref{1}$-$\eqref{2}$ has a positive global solution, which coincides with the Fujita critical exponent result in \cite[Remark 1.6]{RY22}.
\end{cor}

In the cases of compact or exponential volume growth Lie groups (and power nonlinearities), the notion of Fujita exponent becomes irrelevant, the fact that was already observed in \cite{RY22}.

\begin{cor}[Groups of exponential growth]
    Let $\G$ have an exponential volume growth, and let  $\varphi(t)\equiv 1$ and $f(v)=v^{p}$ with $p>1$.
    If $\eqref{1}$-$\eqref{2}$ does not have global positive solutions, 
    condition \eqref{cond2exp} implies that we must have
    $\int_1^\infty \tau^{\frac{\nu}{2}(1-p)}d\tau=\infty$ for all $\nu\geq d$. Hence we must have $1<p\leq 1+\frac{2}{\nu}$ for all $\nu\geq d$, which is impossible. 
    
    Consequently, we can conclude that for any $p>1$, if 
     $\varphi(t)\equiv 1$ and $f(v)\leq K v^{p}$ for some $K>0$,
    $\eqref{1}$-$\eqref{2}$ has a global positive solution, which coincides with the conclusion in  \cite[Theorem 1.5]{RY22}.
\end{cor}

\begin{cor}[Compact Lie groups]
    Let $\G$ be a compact Lie group. In this case we have $D=0$, and condition \eqref{cond2exp} reduces to the property that $f(\theta)\int_1^\infty \varphi(t) d\tau=\infty.$ Since $f(\theta)>0$ for all $\theta>0$, we can conclude that if $\eqref{1}$-$\eqref{2}$ does not have global positive solutions, then we must have $\varphi\not \in L^1[1,\infty)$. In the case $\varphi(t)\equiv 1$ and $f(v)=v^{p}$, this recovers the conclusion in \cite[Remark 1.6]{RY22}.
    
    Alternatively, if $\varphi\in L^1[0,\infty)$, then $\eqref{1}$-$\eqref{2}$ has a global positive solution, for any nonlinearity $f(v)$, satisfying assumptions of Theorem \ref{thm1}.
\end{cor}

Before proceeding with the proof of Theorem \ref{thm1}, let us first prove an estimate for the $L^{\infty}(\G)$-norm of $e^{t \L}u_0$, when $0< u_0 \in C_{0}(\G)$, showing that the norms in \eqref{unimodcond1} are well-defined. For such $u_0$, using \eqref{expr.u0}, we get
\begin{equation*}
    \begin{split}
    e^{t\L}u_0(x) &= \int_{\G}p_t(y^{-1}x)u_0(y)dy\\& 
    \leq \|u_{0}\|_{L^{\infty}(\G)}\int_{\G}p_{t}(y^{-1}x)dy\\&
    =\|u_{0}\|_{L^{\infty}(\G)}\,,
    \end{split}
\end{equation*}
and so we have proved that 
\begin{equation}
    \label{3.5}
    \|e^{t \L}u_0\|_{L^{\infty}(\G)}\leq \|u_{0}\|_{L^{\infty}(\G)}.
\end{equation}
\begin{proof}[Proof of Theorem \ref{thm1}]
 First we prove \eqref{unimodcond1} for both cases (i) and (ii).  We will proceed by contradiction. Assume that there exists $w>0$ such that
   \begin{equation}
       \label{Z}
       Z:=\int_{0}^{\infty}\varphi\left(\tau\right)\frac{f(\|e^{\tau\L}w\|_{L^{\infty}(\G)})}{\|e^{\tau\L}w\|_{L^{\infty}(\G)}}d\tau <\infty\,,
   \end{equation}
   where $w>0$ is a positive function such that $w\in C_{0}(\G)\cap L^{1}(\G)$. We will show that the global solution to \eqref{1}-\eqref{2} exists for some function $0<u_{0}\in C_{0}(\G)\cap L^{1}(\G)$, giving a contradiction. 
   We define $u_0:=\lambda w$ and we also choose $\lambda    \in (0,1)$ such that 
\begin{equation}\label{ocenz}
   \sqrt{\lambda} \Vert e^{t\L}w \Vert_{L^{\infty}(\G)}< 1\,.
\end{equation}
Note that condition \eqref{maj} implies that for each $\epsilon>0$, there exists $\delta>0$ such that if $0<\sqrt{\lambda}(1+Z)<\delta$, then  
$\frac{f_{M} (\sqrt{\lambda} (1+Z))}{\sqrt{\lambda} (1+Z)}<\epsilon,$ so that
\begin{equation}\label{epsilon}
0\leq \frac{f_{M} (\sqrt{\lambda} (1+Z))}{\lambda}<(1+Z)\frac{\epsilon}{\sqrt{\lambda}}\,.\end{equation}

   Let us now define the auxiliary sequence of positive functions $\{v_k\}_{k\geq0}$ on $\mathbb{R}_{+}\times \G$ by 
\begin{equation*}
  v_{k}(t,x):=
  \begin{cases}
  e^{t\L}u_{0}(x)\,, \quad \text{if}\quad k=0,\\
  e^{t\L}u_{0}(x)+\int_{0}^{t}\varphi\left(\tau\right)e^{(t-\tau)\L}f(v_{k-1}(\tau,x))d\tau \,,\quad \text{if}\quad k\geq 1\,.
  \end{cases}
\end{equation*}
Using mathematical induction we will show that for all $k\geq1$ and $(t,x)\in \mathbb{R}_{+}\times \G$, we have 
\begin{equation}
    \label{vk<vo}
    v_k(t,x) \leq (1+Z)e^{t\L}u_0(x)\,.
\end{equation}
Inequality \eqref{vk<vo} is immediate when $k=0$ since we have 
\[v_{0}(t,x)= e^{t\L}u_{0}(x)\leq (1+Z)e^{t\L}u_{0}(x)\]for all $(t,x)\in \mathbb{R}_{+}\times \G$.  Let $\zeta \in \mathbb{N}$, and assume that
\begin{equation}\label{vzeta1}
    v_{\zeta}(t,x) \leq (1+Z)e^{t\L}u_0(x).
\end{equation}
Using the fact that $\frac{f(v)}{v}$ is a non-decreasing function (and hence also $f(v)=\frac{f(v)}{v}v$ is non-decreasing), and noting that $e^{\tau\L}u_0(x)\not=0$ in view of $u_0>0$ and the positivity of the heat kernel,
we get 
\begin{equation*}
 \begin{split}
     v_{\zeta+1}(t,x)&=e^{t\L}u_0(x)+\int_{0}^{t}\varphi(\tau)e^{(t-\tau)\L}f(v_{\zeta}(\tau, x))d\tau\\&
     \stackrel{\eqref{vzeta1}}\leq  e^{t\L}u_0(x)+\int_{0}^{t}\varphi(\tau)e^{(t-\tau)\L}f((1+Z)e^{\tau\L}u_0(x))d\tau\\&
     =e^{t\L}u_0(x)+\int_{0}^{t}\varphi(\tau)e^{(t-\tau)\L}\left[e^{\tau\L}u_0(x)\frac{f((1+Z)e^{\tau\L}u_0(x))}{e^{\tau\L}u_0(x)}\right] d\tau\\&
     \leq e^{t\L}u_0(x)+e^{t\L}u_{0}(x)\int_{0}^{t}\varphi(\tau)\frac{f((1+Z)\|e^{\tau\L}u_0\|_{L^{\infty}(\G)})}{\|e^{\tau\L}u_0\|_{L^{\infty}(\G)}} d\tau.
 \end{split}   
\end{equation*}
Hence for  $u_0(x)=\lambda w(x)$, with $\lambda>0$ as above,  by the definition of the majorant function and the non-decreasing property of the functions $\frac{f(v)}{v}$ and $f(v)$ we get 
\begin{equation}
    \begin{split}
       &v_{\zeta+1}(t,x) \leq e^{t\L}u_0(x)+e^{t\L}u_{0}(x)\int_{0}^{t}\varphi(\tau)\frac{f((1+Z)\|e^{\tau\L}u_0\|_{L^{\infty}(\G)})}{\|e^{\tau\L}u_0\|_{L^{\infty}(\G)}} d\tau\\&
=e^{t\L}u_0(x)+e^{t\L}u_{0}(x)\int_{0}^{t}\varphi(\tau)\frac{f(\lambda(1+Z)\|e^{\tau\L}w\|_{L^{\infty}(\G)})}{\lambda\|e^{\tau\L}w\|_{L^{\infty}(\G)}} d\tau\\&
\stackrel{\eqref{maj1},\eqref{ocenz}}\leq e^{t\L}u_0(x)+\frac{f_{M}(\sqrt{\lambda}(1+Z))}{\sqrt{\lambda}}e^{t\L}u_{0}(x)\int_{0}^{t}\varphi(\tau)\frac{f(\sqrt{\lambda}\|e^{\tau\L}w\|_{L^{\infty}(\G)})}{\sqrt{\lambda}\|e^{\tau\L}w\|_{L^{\infty}(\G)}} d\tau\\&
\stackrel{\lambda <1}\leq  e^{t\L}u_0(x)+\frac{\sqrt{\lambda}f_{M}(\sqrt{\lambda}(1+Z))}{\lambda}e^{t\L}u_{0}(x)\int_{0}^{\infty}\varphi(\tau)\frac{f(\|e^{\tau\L}w\|_{L^{\infty}(\G)})}{\|e^{\tau\L}w\|_{L^{\infty}(\G)}} d\tau\\&
\stackrel{\eqref{epsilon}}\leq e^{t\L}u_0(x)+\epsilon(1+Z) e^{t\L}u_{0}(x)\int_{0}^{\infty}\varphi(\tau)\frac{f(\|e^{\tau\L}w\|_{L^{\infty}(\G)})}{\|e^{\tau\L}w\|_{L^{\infty}(\G)}} d\tau\\&
\stackrel{\eqref{Z}}\leq (1+ Z) e^{t\L}u_{0}(x)\,,
    \end{split}
\end{equation}
where for the last inequality we have chosen $\epsilon$ such that $\epsilon(1+Z)<1$, and we have proved our claim, i.e., that \eqref{vk<vo} holds true for any $k \geq 0$. 

Next, we will show, again by mathematical induction, that $v_{k}\leq v_{k+1}$ for all $k\geq 0$. To this end we note that the positivity $f$ and $\varphi$, give \[v_{0}(t,x)=e^{t\L}u_0(x)\leq e^{t\L}u_0(x)+\int_{0}^{t}\varphi\left(\tau\right)e^{(t-\tau)\L}f(v_{0}(\tau,x))d\tau=v_{1}(t,x)\,.\]   Assume now that for some fixed $\zeta\in\mathbb{N}$ the inequality $v_{\zeta}(t,x)\geq v_{\zeta-1}(t,x)$  holds true. Then, since $f$ is a non-decreasing function, we have 
\begin{equation}
    \begin{split}
     v_{\zeta+1}(t,x)&= e^{t\L}u_{0}(x)+\int_{0}^{t}\varphi\left(\tau\right)e^{(t-\tau)\L}f(v_{\zeta}(\tau,x))d\tau \\&
     \geq  e^{t\L}u_{0}(x)+\int_{0}^{t}\varphi\left(\tau\right)e^{(t-\tau)\L}f(v_{\zeta-1}(\tau,x))d\tau =v_{\zeta}(t,x),
    \end{split}
\end{equation}
and the latter proves our claim that the sequence $\{v_k\}$ is non-decreasing with respect to $k$. 
 
Summarising the above, the monotonicity of $\{v_k\}_{k\geq0}$, together with the upper bound for each $\{v_k\}_{k\geq0}$ given by \eqref{vk<vo}, imply, as an application of the monotone convergence theorem, that the limit $\lim\limits\limits_{k \rightarrow \infty}v_k(t,x)$ exists globally. Hence, by the definition of the sequence $\{v_k\}_{k\geq0}$ we get 
\[
\lim_{k \rightarrow \infty}v_k(t,x)=e^{t\L}u_0(x)+\int_{0}^{t}\varphi(\tau)e^{(t-\tau)\L}f(\lim_{k \rightarrow \infty}v_k(t,x))d\tau\,.
\]
On the other hand, by \eqref{expr.u0} we get that $\lim\limits_{k \rightarrow \infty}v_k(t,x)=u(t,x)$, so that by taking the limit as $k \rightarrow \infty$ of the inequality \eqref{vk<vo} we get 
\begin{equation}
    \label{upper.l.u}
    u(t,x)\leq (1+Z)e^{t\L}u_0(x)\,,
\end{equation}
and the latter implies that the solution $u(t,x)$ exists globally and we proved the first part.

Let us show \eqref{cond2poly} under conditions of Part (i). Take  $0<u_{0}\in C_{0} (\G)\cap L^{1}(\G)$, then by \eqref{unimodcond1} we have
    \begin{eqnarray}\label{a-b1}
        +\infty & = & \int_{0}^{\infty}\varphi(\tau)\frac{f(\|e^{\tau\L}u_{0}\|_{L^{\infty}(\G)})}{\|e^{\tau\L}u_{0}\|_{L^{\infty}(\G)}}d\tau \nonumber \\
        & = & \int_{0}^{1}\varphi(\tau)\frac{f(\|e^{\tau\L}u_{0}\|_{L^{\infty}(\G)})}{\|e^{\tau\L}u_{0}\|_{L^{\infty}(\G)}}d\tau+ \int_{1}^{\infty}\varphi(\tau)\frac{f(\|e^{\tau\L}u_{0}\|_{L^{\infty}(\G)})}{\|e^{\tau\L}u_{0}\|_{L^{\infty}(\G)}}d\tau\,.
    \end{eqnarray}
   Since $\int_{\G}p_{t}(x)dx=1$ with $0<u_{0}\in C_{0} (\G)\cap L^{1}(\G)$, by \eqref{3.5} we have
    \begin{equation*}
        \begin{split}
            \|e^{t\L}u_{0}\|_{L^{\infty}(\G)}\leq \|u_{0}\|_{L^{\infty}(\G)},\,\,\,\,\,\text{for\,\,\,all}\,\,t>0.
        \end{split}
    \end{equation*}  
Combining the last fact with the continuity and positivity of both $f$ and $\varphi$, and $\frac{f(v)}{v}$ being non-decreasing, we obtain that 
 $$\int_{0}^{1}\varphi(\tau)\frac{f(\|e^{\tau\L}u_{0}\|_{L^{\infty}(\G)})}{\|e^{\tau\L}u_{0}\|_{L^{\infty}(\G)}}d\tau\leq \frac{f(\|u_{0}\|_{L^{\infty}(\G)})}{\|u_{0}\|_{L^{\infty}(\G)}}\int_{0}^{1}\varphi(\tau)d\tau<+\infty.$$
  Hence, using \eqref{a-b1} we get that
    \begin{eqnarray}
        \label{a-b41}
        +\infty & = &  \int_{1}^{\infty}\varphi(\tau)\frac{f(\|e^{\tau\L}u_{0}\|_{L^{\infty}(\G)})}{\|e^{\tau\L}u_{0}\|_{L^{\infty}(\G)}}d\tau.
    \end{eqnarray}
    Now, to estimate the norm $ \|e^{t\L}u_{0}\|_{L^{\infty}(\G)}$, using \eqref{asymptpoly} and \eqref{def.heat}, for all $t>0$ we have 
    \begin{equation}\label{semigroup.est1}
    \begin{split}
          e^{t\L}u_{0}(x) &=\int_{\G}p_{t}(y^{-1}x)u_{0}(y)dy\\&
        \leq C \V^{-1}\int_{\G}e^{-\frac{c\rho^2(y^{-1}x)}{t}}u_{0}(y)dy\\&
        \leq   C \V^{-1}  \int_{\G}u_{0}(y)dy.
    \end{split}
    \end{equation}
   Inequality \eqref{semigroup.est1} implies that for all $t\geq 1$
    \begin{equation*}\label{h.last1}
    \begin{split}
         \|e^{t\L}u_{0}\|_{L^{\infty}(\G)}&\leq C \V^{-1} \int_{\G}u_{0}(y)dy\\&
         \leq C t^{-\frac{D}{2}}\int_{\G}u_{0}(y)dy\\&
         =\theta t^{-\frac{D}{2}},
    \end{split}
    \end{equation*}
    where $D$ is the global dimension and we have set $\theta=C \int_{\G}u_{0}(y)dy.$
 
   Hence, since $0 \leq \varphi$, using the monotonicity of the function $\frac{f(v)}{v}$ we get
    \begin{equation}       
         +\infty =\int_{1}^{\infty}\varphi(\tau)\frac{f(\|e^{\tau\L}u_{0}\|_{L^{\infty}(\G)})}{\|e^{\tau\L}u_{0}\|_{L^{\infty}(\G)}}d\tau \leq \int_{1}^{\infty}\varphi(\tau)\frac{f(\theta \t^{-\frac{D}{2}})}{\theta \t^{-\frac{D}{2}}}d\tau\,,
    \end{equation}
    and the latter implies that 
    \[
 \int_{1}^{\infty}\varphi(\tau)\frac{f(\theta \t^{-\frac{D}{2}})}{ \t^{-\frac{D}{2}}}d\tau=+\infty.
    \]
 Since this is true for any $0<u_{0}\in C_{0}(\G)\cap L^{1}(\G)$, and hence for any $\theta>0$, we complete the proof of \eqref{cond2poly}.

 Let us prove \eqref{cond2exp} for the case of exponential growth in Part (ii). Similarly
 to the previous case we have
   \begin{eqnarray*}
        +\infty & = & \int_{0}^{\infty}\varphi(\tau)\frac{f(\|e^{\tau\L}u_{0}\|_{L^{\infty}(\G)})}{\|e^{\tau\L}u_{0}\|_{L^{\infty}(\G)}}d\tau \nonumber \\
        & = & \int_{0}^{1}\varphi(\tau)\frac{f(\|e^{\tau\L}u_{0}\|_{L^{\infty}(\G)})}{\|e^{\tau\L}u_{0}\|_{L^{\infty}(\G)}}d\tau+ \int_{1}^{\infty}\varphi(\tau)\frac{f(\|e^{\tau\L}u_{0}\|_{L^{\infty}(\G)})}{\|e^{\tau\L}u_{0}\|_{L^{\infty}(\G)}}d\tau\,.
    \end{eqnarray*}
    By combining the last fact with $\|e^{\L}u_{0}\|_{L^{\infty}(\G)}\leq \|u_{0}\|_{L^{\infty}(\G)}$ and using \eqref{asymptexp}, we obtain
       \[
 \int_{1}^{\infty}\varphi(\tau)\frac{f(\theta \t^{-\frac{\nu}{2}})}{ \t^{-\frac{\nu}{2}}}d\tau=+\infty,
    \]
    completing the proof.
\end{proof}

\section{Particular case of $\H$}\label{SEC:4}

Our setting in this section is that of the Heisenberg group $\H$. After giving an elementary  description of the Heisenberg group $\H \cong \mathbb{R}^{2n+1}$ (of any topological dimension) we recall results on $\H$ that are necessary for the subsequent  analysis. 

The group law on $\H$ reads as follows

\begin{equation*}
\begin{split}
x \circ x'=(\xi+\xi',\zeta +\zeta'+2\text{Im}(\xi\cdot\xi')),\,\,x=(\xi,\zeta)\in\mathbb{C}^{n}\times\mathbb{R},
\end{split}
\end{equation*}
and the canonical left-invariant vector fields that span its Lie algebra can be computed as 
$$
X_{i}=\frac{\partial}{\partial g_{i}}+2v_{i}\frac{\partial}{\partial \zeta}, \,\,\, Y_{i}=\frac{\partial}{\partial v_{i}}-2g_{i}\frac{\partial}{\partial \zeta},\,\, T=\frac{\partial}{\partial \zeta},\,\,\,\quad j=1,\ldots,n\,,\,\,\, \xi=g+\text{i}v.
$$
The Heisenberg group $\H$ is a nilpotent Lie group of step 2 since we have
\[
[X_i,Y_i]=T\,,\quad \forall i=1,\cdots,n\,,
\]
and the the second-order differential operator 
$$\mathcal{L}_{\H}=\sum_{i=1}^{n}\left(X_{i}^{2}+Y_{i}^{2}\right)$$
is then the hypoelliptic analog of the Laplace-Beltrami operator in the Euclidean case, that is commonly called the sub-Laplacian on $\H$. Since the vector fields $X_i,Y_i$, $i=1,\cdots,n$, belong to the first stratum of the Lie algebra of $\H$, $\H$ is a stratified Lie group, and so also a homogeneous Lie group. 

Li in \cite{Li07} refined the estimates \eqref{asympt} on any nilpotent Lie group for the case of the Heisenberg group $\H$. In particular he obtained the optimal decay estimates of the heat kernel $p_t$ on $\H$ for the case where $t=1$.  The latter  reads as follows:
\begin{equation}\label{Lie.est1}
    A^{-1}P(n;\|\x\|,\rho(\x,\zeta))\leq p_{1}(\x,\zeta)\leq A P(n;\|\x\|,\rho(\x,\zeta))\,,\quad A>1\,,
\end{equation}
where $\rho$ stands, as in the general case, for the Carnot-Carath\'{e}odory distance from the origin, and $\|\cdot\|$ stands for the Euclidean distance (in this case applied to the vector lying in the first stratum of $\H$), and we have set $P(n;y,w)=e^{-\frac{w^{2}}{4}}\frac{(1+w^{2})^{n-1}}{(1+y w)^{n-\frac{1}{2}}}$ for all $y,w \geq 0$. Inequality \eqref{Lie.est1} can also be written in the following form 
\[
p_{1}(\x,\zeta) \asymp P(n;\|\x\|,\rho(\x,\zeta))\,.
\]
Note also that the fundamental solution to the heat equation on $\H$ has the following integral form
\begin{equation}\label{explicit}
    p_{t}(x)=p_{t}(\xi,\tau)=\frac{1}{2(4\pi t)^{n+1}}\int_{\mathbb{R}}e^{\frac{\lambda}{4t}\left(\tau-\|\xi\|^{2}\coth \lambda\right)}\left(\frac{\lambda}{\sinh \lambda}\right)^{n}d\lambda\,,
\end{equation}
see \cite{H76}, \cite{G77}, $p_t(x)dx$ is a probability measure on $\H$, see \cite{N96}, and  for all $t>0,$ $x \in \H$, we have $p_t(\x,\zeta)=t^{-n-1}p_1\left(\frac{\x}{\sqrt{t}}, \frac{\zeta}{t}\right)$. Hence, using \eqref{Lie.est1} and the homogeneity of $\rho$, $\|\cdot\|$ (with respect to the dilations), we get 
\begin{equation}\label{asymp2}
    p_{t}(\x,\zeta)\asymp t^{-n-1} P\left(n;\left\|\frac{\xi}{\sqrt{t}}\right\|,\frac{\rho(\x,\zeta)}{\sqrt{t}}\right)\,,
\end{equation}
or equivalently 
\begin{equation}
    \label{asympt23}
     p_{t}(\x,\zeta)\asymp t^{-n-1} e^{-\frac{\rho^2(\x,\zeta)}{4t}}\frac{\left(1+\frac{\rho^2(\x,\zeta)}{t} \right)^{n-1}}{\left(1+\frac{\|\x\|\rho(\x,\zeta)}{t} \right)^{n-\frac{1}{2}}}\,. 
\end{equation}

For completeness, we recall that  the initial value problem \eqref{1}-\eqref{2} in the case of the Heisenberg group $\H$ reads as below:
\begin{equation}\label{1h}
    u_{t}(t,x)-\Lh u(t,x)=\varphi(t)f(u(t,x)),\,\,\,\,\,(t,x)\in \mathbb{R}_{+}\times \H,
\end{equation}
with initial data 
\begin{equation}\label{2h}
    u(0,x)=u_{0}(x),\,\,\,\,\,x\in \H.
\end{equation}

From Theorem \ref{thm1}, we have the following consequence on $\H$:
\begin{cons}\label{cons1}
    Let $\H$ be the Heisenberg group of homogeneous dimension $Q=2n+2$.  Assume that  $f$ is a  continuous, non-negative function with $f(0)=0$ and $f(v)>0$ when $v>0$,  such that \eqref{maj} holds
   and the mapping  $v\mapsto\frac{f(v)}{v}$ is non-decreasing. Let also  $0\leq\varphi\in L^{1}_{\text{loc}}[0,+\infty)$. If, for all $0<u_{0}\in C_{0} (\H)\cap L^{1}(\H)$, \eqref{1h}-\eqref{2h} does not have a global solution, then  for all $0<w\in C_{0}(\H)\cap L^{1}(\H)$ we have
\begin{equation}\label{Hncondthm}
   \int_{0}^{\infty}\varphi(\tau)\frac{f(\|e^{\tau\LLL}w\|_{L^{\infty}(\H)})}{\|e^{\tau\LLL}w\|_{L^{\infty}(\H)}}d\tau=+\infty.
\end{equation}

Moreover, in this case, for any $\varepsilon>0$,  we  also have
  \begin{equation}\label{Hncond1}
      \int_{1}^{\infty}\varphi(\tau)\tau^{n+1} f(\varepsilon\tau^{-n-1})d\tau=+\infty.
  \end{equation}

\end{cons}
We will now give a partial converse to this statement, which gives the exact converse in the case of $n=1$.
\begin{thm}\label{onhn}
     Let $\H$ be the Heisenberg group of homogeneous dimension $Q=2n+2$.  Assume that  $f$ is a  continuous, convex, non-negative function with $f(0)=0$ and $f(v)>0$ when $v>0$,  such that \eqref{min} holds. Let also $0\leq\varphi\in C[0,+\infty)$. If the condition 
   \begin{equation}\label{Hncond}
       \int_{1}^{\infty}\varphi(\tau)\tau^{2n}f(\varepsilon \tau^{-2n})d\tau=+\infty
   \end{equation} 
   holds for every $\varepsilon >0$, then  there exists no global solution of \eqref{1h}-\eqref{2h}  for any data $0<u_{0}\in C_{0} (\H)\cap L^{1}(\H)$.

\end{thm}

Combining Theorem \ref{thm1} and Theorem \ref{onhn}, we thus obtain the necessary and sufficient conditions in the case of 
the Heisenberg group $\Hh$ of homogeneous dimension $Q=4$. 

\begin{cor}\label{cor1}
       Let $\Hh$ be the Heisenberg group of homogeneous dimension $Q=4$.  Assume that  $f$ is a  continuous, convex, non-negative function with $f(0)=0$ and $f(v)>0$ for $v>0$,  such that \eqref{min} and \eqref{maj} hold. Let also $0\leq\varphi\in C[0,+\infty)$. Then  the following statements are equivalent:
 \begin{itemize}
 \item[(a)] $\int_{0}^{\infty}\varphi(\tau)\frac{f(\|e^{\tau\LL}w\|_{L^{\infty}(\Hh)})}{\|e^{\tau\LL}w\|_{L^{\infty}(\Hh)}}=+\infty,$ for all $0<w\in C_{0}(\Hh)\cap L^{1}(\Hh)$;
      \item[(b)] $\int_{1}^{\infty}\varphi(\tau)\tau^{2}f(\varepsilon \tau^{-2})d\tau=+\infty$, for all $\varepsilon >0$;
       \item [(c)] For all $0<u_{0}\in C_{0} (\Hh)\cap L^{1}(\Hh)$, there exists no global solution of \eqref{1h}-\eqref{2h} with initial data $u_{0}$.
  \end{itemize}
\end{cor}

\begin{rem}
 On the $n$-dimensional Euclidean space $\mathbb{R}^{n}$, we also have equivalence of (a), (b) and (c), see \cite{CH22}.    
\end{rem}
    
\begin{rem}
   In the particular case when $n=1,$ $\varphi=1$ and $f(v)=v^p$ in (b), if $1<p\leq p_{c}=\frac{3}{2}$, the solution \eqref{1h}-\eqref{2h} blows-up. For the case $n=1$, our critical exponent $p_{c}=\frac{3}{2}$ coincides with the Fujita critical exponent in \cite{P98,RY22} ($p^{*}=1+\frac{2}{Q}\stackrel{n=1}=\frac{3}{2}$).

    In the general case when $n>1$, $1<p\leq p^{*}=1+\frac{1}{2n},$ from \eqref{Hncond} we  compute
\begin{equation*}
\varepsilon^{p}  \int_{1}^{\infty}\tau^{2n} \tau^{-2np}d\tau=+\infty.    
\end{equation*}
From \cite{RY22} and \cite{P98}, the critical exponent equals to $p_{c}=1+\frac{1}{n+1}$. Comparing this with our result, we have $p^{*}\leq p_{c}$, indicating that our result is not optimal. 

We {\bf conjecture} that on general unimodular Lie groups of polynomial growth, the optimal condition is given by \eqref{cond2poly}, which in the case of the Heisenberg group $\mathbb H^n$ reduces to
\begin{equation}
    \int_{1}^{\infty}\varphi(\tau)\tau^{n+1}f(\varepsilon \tau^{-n-1})d\tau=+\infty.
\end{equation}
\end{rem}
\begin{rem}
    Let $f(v)=v^{p}+v^{q}$  and $\varphi(t)=t^{r}+t^{s}$ where $p\geq q>1$, $t>0$,  and $r\geq s\geq -1$. Then by definitions of minorant and majorant functions, we have 
    \begin{equation*}
        f_{M}(v)=\max\left\{\frac{v^{p}+v^{q}}{2}, v^{q}\right\}\,\,\,\,\text{and}\,\,\,\,f_{m}(v)=\min\left\{\frac{v^{p}+v^{q}}{2}, v^{q}\right\}.
    \end{equation*}
    By using the last fact, we get
    \begin{equation*}
        \lim_{v\rightarrow 0^{+}}\frac{f_{M}(v)}{v}=0,\,\,\,\,\text{and}\,\,\,\,\int_{1}^{+\infty}\frac{dv}{f_{m}(v)}<+\infty.
    \end{equation*}
    Checking \eqref{Hncond1} with $p\leq 1+\frac{s+1}{2n}$ (that is, $s\geq 2n(p-1)-1$), we compute
    \begin{equation*}
    \begin{split}
         &\int_{1}^{+\infty}(\tau^{r}+\tau^{s})\tau^{2n}(\varepsilon^{p} \tau^{-2n p}+\varepsilon^{q} \tau^{-2n q})d\tau \\&
    =\int_{1}^{+\infty}\left[\varepsilon^{p}(\tau^{r+2n-2np}+\tau^{s+2n-2np})+\varepsilon^{q}(\tau^{r+2n-2nq}+\tau^{s+2n-2nq})\right]d\tau\stackrel{s\geq 2n(p-1)-1}=+\infty.
    \end{split}
    \end{equation*}
    Then for all $0<u_{0}\in C_{0} (\H)\cap L^{1}(\H)$, there exists no global solution of \eqref{1h}-\eqref{2h} with initial data $u_{0}$. In the particular case $n=1$, critical exponent equals to $p_{c}=\frac{s+1}{2}+1$. 
    
    On the $n$-dimensional Euclidean space $\mathbb{R}^{n}$, this case  was considered in \cite{LP14}. 
\end{rem}
\begin{rem}
Here, we note that from the convexity of function $f(v)$ with  $f(0)=0$ and $f(v)>0$ when $v>0$, we have that $v\mapsto \frac{f(v)}{v}$ is a non-decreasing function. From  the convexity of the function $f$, we have that $F(v,v')=\frac{f(v)-f(v')}{v-v'}$ is a non-decreasing function with respect to $v$ and for every $v'$. In addition,  if we assume that $f(0)=0$ and $f(v)>0$ when $v>0$, we have $F(v,0)=\frac{f(v)}{v}$, that is,  the $v\mapsto \frac{f(v)}{v}$ is a non-decreasing function. 
\end{rem}
 Firstly, here we present a preliminary lemma for the heat kernel on $\mathbb{H}^{n}$.
\begin{lem}\label{Lem1}
    Let $\H$ be the Heisenberg group with homogeneous dimension $Q=2n+2$ and let  $t_{1}\geq t_{2}>0$. Then, we have
    \begin{equation}\label{lemeq}
        \frac{p_{t_{1}}(x)}{p_{t_{2}}(x)}\geq A^{-2}\left(\frac{t_{2}}{t_{1}}\right)^{2n},\,\,\,\,\textit{for\,\, any}\,\,\,\,\, x\in \H,
    \end{equation}
    for some $A>1$ independent of $t_{1},t_{2}$ and $x$.
\end{lem}
\begin{proof}[Proof of Lemma \ref{Lem1}]
 By using the optimal estimate of the heat kernel \eqref{asympt23}, we compute, for some $A>1$,
    \begin{equation*}
        \begin{split}
             p_{t_{1}}(x)&\geq A^{-1}t_{1}^{-n-1} e^{-\frac{\rho^2(x)}{4t_{1}}}\frac{\left(1+\frac{\rho^2(x)}{t_{1}} \right)^{n-1}}{\left(1+\frac{\|\x\|\rho(x)}{t_{1}} \right)^{n-\frac{1}{2}}}\\&
             \geq A^{-1}t_{1}^{-n-1} t_{2}^{n+1}t_{2}^{-n-1}e^{-\frac{\rho^2(x)}{4t_{2}}}\frac{\left(1+\frac{\rho^2(x)}{t_{1}} \right)^{n-1}}{\left(1+\frac{\|\x\|\rho(x)}{t_{2}} \right)^{n-\frac{1}{2}}}\\&
             \geq A^{-1}t_{1}^{-n-1} t_{2}^{n+1}t_{2}^{-n-1} e^{-\frac{\rho^2(x)}{4t_{2}}}\frac{\left(\frac{t_{2}}{t_{1}}+\frac{\rho^2(x)}{t_{1}} \right)^{n-1}}{\left(1+\frac{\|\x\|\rho(x)}{t_{2}} \right)^{n-\frac{1}{2}}}\\&
             =A^{-1}\frac{t^{n+1}_{2}}{t^{n+1}_{1}}\frac{t^{n-1}_{2}}{t^{n-1}_{1}}t_{2}^{-n-1}e^{-\frac{\rho^2(x)}{4t_{2}}}\frac{\left(1+\frac{\rho^2(x)}{t_{2}} \right)^{n-1}}{\left(1+\frac{\|\x\|\rho(x)}{t_{2}} \right)^{n-\frac{1}{2}}}\\&
             \geq A^{-2}\frac{t^{2n}_{2}}{t^{2n}_{1}}p_{t_{2}}(x),
        \end{split}
    \end{equation*}
    completing the proof.
\end{proof}

\begin{proof}[Proof of Theorem \ref{onhn}]

    Let us now  prove the non-existence result. 
    By the property of the semigroup $e^{t\LLL}$ we get
    \begin{equation}\label{b-c1'}
        \begin{split}
e^{t\LLL}&u(t,x)\stackrel{\eqref{expr.u0}}=e^{t\LLL}\left[e^{t\LLL}u_0(x)+\int_{0}^{t} e^{(t-\tau)\LLL}\varphi(\tau)f(u(\tau,x))\,d\tau\right]\\& 
=e^{2t\LLL}u_0(x)+\int_{0}^{t} e^{(2t-\tau)\LLL}\varphi(\tau)f(u(\tau,x))\,d\tau\\&
\stackrel{\eqref{def.heat}}=\int_{\H}p_{2t}(y^{-1}x)u_0(y)dy+\int_{0}^{t}\int_{\H}p_{2t-\tau}(y^{-1}x)\varphi(\tau)f(u(\tau,y))dyd\tau\,.
\end{split}
\end{equation}

Using \eqref{b-c1'} and the optimal decay estimate for the heat kernel \eqref{asympt23},  we compute
\begin{equation}\label{b-c2'}
    \begin{split}
    e^{t\LLL}u(t,x)& \geq \frac{(2t)^{-n-1}}{A} \int_{\H}\frac{e^{-\frac{\rho^{2}(y^{-1}x)}{8t}}\left(1+\frac{\rho^2(y^{-1}x)}{2t} \right)^{n-1}}{\left(1+\frac{\|(y^{-1}x)_1\|\rho(y^{-1}x)}{2t}\right)^{n-\frac{1}{2}}}u_{0}(y)dy\\&
    +\int_{0}^{t}\int_{\H}p_{2t-\tau}(y^{-1}x)\varphi(\tau)f(u(\tau,y))dyd\tau \\&
:=\gamma_{1}(t,x)+\gamma_{2}(t,x)\,.
        \end{split}
    \end{equation}
Assume that the positive solution $u$ exists globally. By fixing $x\in \H$, the set $B_x:=\{y \in \H\,:\rho(y^{-1}x)<1 \}$ is compact and $u_0>0$, hence there exists a constant $C_{x}$,   such that $u_0(y)>C_{x}>0$ for all $y \in B_x$. Hence, for $t \geq 1$ we have
    \begin{eqnarray}\label{b-c3'}
        \gamma_{1}(t,x) & \geq & \frac{C_{x}}{A(2t)^{n+1}} \int_{B_x} \frac{e^{-\frac{\rho^{2}(y^{-1}x)}{8t}}\left(1+\frac{\rho^2(y^{-1}x)}{2t} \right)^{n-1}}{\left(1+\frac{\|(y^{-1}x)_1\|\rho(y^{-1}x)}{2t}\right)^{n-\frac{1}{2}}}dy \nonumber \\
        & \stackrel{t\geq1} \geq & \frac{C_{x}}{A(B+1)^{n-\frac{1}{2}}(2t)^{n+1}} \int_{B_x}  e^{-\frac{\rho^{2}(y^{-1}x)}{8}}dy \\
        & = & \frac{C_{x}(2t)^{n-1}}{A(1+B)^{n-\frac{1}{2}}(2t)^{n+1}(2t)^{n-1}} \int_{B_x}  e^{-\frac{\rho^{2}(y^{-1}x)}{8}}dy\nonumber \\
        & \stackrel{t\geq1} \geq &\frac{2^{n-1}C_{x}}{A(1+B)^{n-\frac{1}{2}}(2t)^{2n}} \int_{B_x}  e^{-\frac{\rho^{2}(y^{-1}x)}{8}}dy \nonumber\,
    \end{eqnarray}
   for some $B>0$, since by the equivalence of the quasi-norms we have
    \[
    \left(1+\frac{\|(y^{-1}x)_1\|\rho(y^{-1}x)}{2}\right)^{n-\frac{1}{2}} \leq \left(1+\frac{2B\rho^{2}(y^{-1}x)}{2}\right)^{n-\frac{1}{2}} \leq (1+B)^{n-\frac{1}{2}}\,,\quad \forall y \in B_x\,.
    \]
   Therefore, from \eqref{b-c3'} we obtain 
    \begin{equation}
        \label{b-c4'}
        \gamma_{1}(t,x)   \geq   D^{x}_1 t^{-2n}\,,
    \end{equation}
   where  \[D^{x}_1:=\frac{C_{x}}{A(1+B)^{n-\frac{1}{2}}2^{n+1}}\int_{B_x}  e^{-\frac{\rho^{2}(y^{-1}x)}{8}}dy <\infty\,.\]
   
Let us now estimate the term $\gamma_{2}(t,x)$. Notice that, using Lemma \ref{Lem1} with $0\leq \tau\leq t$ (where $t_{2}=\tau$ and $t_{1}=2t-\tau$), we establish
\begin{equation*}
    \begin{split}
\gamma_{2}(t,x) &=\int_{0}^{t}\int_{\H}p_{2t-\tau}(y^{-1}x)\varphi(\tau)f(u(\tau,y))dyd\tau\\&
\stackrel{\eqref{lemeq}}\geq \frac{1}{A^{2}}\int_{0}^{t}\left(\frac{\tau}{2t-\tau}\right)^{2n}\int_{\H}p_{\tau}(y^{-1}x)\varphi(\tau)f(u(\tau,y))dyd\tau\\&
\geq \frac{1}{A^{2}(2t)^{2n}}\int_{0}^{t}\tau^{2n}\int_{\H}p_{\tau}(y^{-1}x)\varphi(\tau)f(u(\tau,y))dyd\tau,        
    \end{split}
\end{equation*}
and by using Jensen's inequality, we get
\begin{equation}\label{b-c5''}
    \begin{split}
\gamma_{2}(t,x) &
\geq \frac{1}{A^{2}(2t)^{2n}}\int_{0}^{t}\tau^{2n}\int_{\H}p_{\tau}(y^{-1}x)\varphi(\tau)f(u(\tau,y))dyd\tau\\&
\geq \frac{1}{A^{2}(2t)^{2n}}\int_{0}^{t}\tau^{2n}\varphi(\tau)f\left(\int_{\H}p_{\tau}(y^{-1}x)u(\tau,y)dy\right)d\tau\\&
=\frac{1}{A^{2}(2t)^{2n}}\int_{0}^{t}\tau^{2n}\varphi(\tau)f\left(e^{t\LLL}u\right)d\tau.       
    \end{split}
\end{equation}
Summarising, by \eqref{b-c5''} and \eqref{b-c4'} we get
\begin{equation}\label{why,blow'}
 e^{t\LLL}u(t,x)
 \geq 
 \gamma_{1}(t,x)+\gamma_{2}(t,x) \geq t^{-2n}\gamma_{3}(t,x)\,, 
\end{equation}
where 
\begin{equation}\label{b-c6'}
    \gamma_{3}(t,x):=D^{x}_1+\frac{1}{2^{2n}A^{2}}\int_{0}^{t}\tau^{2n}\varphi(\tau)f\left(e^{\tau\LLL}u(\tau,x)\right)d\tau.
\end{equation}
    Since $u(t,x)$ is a global solution, $\gamma_{3}(t,x)$ exists globally. Taking derivative with respect to $t$ we obtain

\begin{equation}\label{ocen123'}
    \begin{split}
        \frac{d\gamma_{3}(t,x)}{dt}& =  \frac{1}{2^{2n}A^{2}}t^{2n}\varphi(t)f\left(e^{t\LLL}u(t,x) \right)\\&
    \stackrel{\eqref{why,blow'}} \geq \frac{1}{2^{2n}A^{2}}t^{2n}\varphi(t)f\left[t^{-2n}\gamma_{3}(t,x)\right].
    \end{split}
\end{equation}
Let us now consider  $0<\varepsilon<1$ satisfying the inequality $\varepsilon< D^{x}_{1}$. By the definition of the minorant function $f_m$ and \eqref{maj1} we have
\begin{equation}\label{appl.min'}
    \begin{split}
    f\left[t^{-2n}\gamma_{3}(t,x)\right]&=\frac{f\left(\frac{\varepsilon}{t^{2n}}\frac{\gamma_{3}(t,x)}{\varepsilon}\right)}{f\left(\frac{\varepsilon}{t^{2n}}\right)}f\left(\frac{\varepsilon}{t^{2n}}\right)\\&
    \geq f_{m}\left(\frac{\gamma_{3}(t,x)}{\varepsilon}\right)f\left(\frac{\varepsilon}{t^{2n}}\right)\,,
    \end{split}
\end{equation}
for $0<\varepsilon<1$ and $t>1$. Also, we note that  $f_{m}(v)=0$ at the origin and $f(v)\geq 1$ for $v\geq1$. Hence, by $\varepsilon< D^{x}_{1}$ we compute
\begin{equation*}
    \begin{split}
    \frac{f\left(\frac{\varepsilon}{t^{2n}}\frac{\gamma_{3}(t,x)}{\varepsilon} \right)}{f\left(\frac{\varepsilon}{t^{2n}}\right)}\geq \frac{f\left(\frac{\varepsilon}{t^{2n}}\frac{D_{1}^{x}}{\varepsilon} \right)}{f\left(\frac{\varepsilon}{t^{2n}}\right)}\geq\frac{f\left(\frac{\varepsilon}{t^{2n}}\right)}{f\left(\frac{\varepsilon}{t^{2n}}\right)}=1,
    \end{split}
\end{equation*}
and by taking infimum with respect to $\frac{\varepsilon}{t^{2n}}$ with $0<\varepsilon <1$ and $t>1$, we have 
\begin{equation*}
    f_{m}\left(\frac{\gamma_{3}(t,x)}{\varepsilon}\right)\geq 1,\,\,\,\,\,t>1.
\end{equation*}
Now, a combination of \eqref{ocen123'} together with \eqref{appl.min'} and the last fact yields
\begin{equation}\label{tointeg'}
    \frac{\gamma'_{3}(t,x)}{f_{m}\left(\frac{\gamma_{3}(t,x)}{\varepsilon}\right)}\geq \frac{1}{2^{2n}A^{2}}t^{2n}\varphi(t)f\left(\varepsilon t^{-2n}\right),\,\,\,\,\,t>1.
\end{equation}

Now, integrating  inequality \eqref{tointeg'} over $(1,t)$, we get
\begin{equation}\label{ocen1234'}
\begin{split}
    G(\gamma_{3}(1,x))-G(\gamma_{3}(t,x))\geq \frac{1}{2^{2n}A^{2}\varepsilon} \int_{1}^{t}\tau^{2n}\varphi(\tau)f\left(\varepsilon \tau^{-2n}\right)d\tau,
\end{split}
\end{equation}
where $G(r):=\int_{\frac{r}{\varepsilon}}^{+\infty}f^{-1}_{m}(r_{1})dr_{1}$ and from \eqref{min} we have $G(r)<+\infty$ . It is easy to see that that the function $G$ is bijective, positive and decreasing, and we have $G(r)\rightarrow 0$ as $r\rightarrow +\infty.$ Note that by  \eqref{Hncond} we have
\[
\int_{1}^{\infty}\tau^{2n}\varphi(\tau)f\left(\varepsilon \tau^{-2n}\right)d\tau=+\infty\,,
\]
for all $\varepsilon>0$. Now, if we denote $h=h(t):=  \frac{1}{2^{2n}A^{2}\varepsilon}\int_{1}^{t}\tau^{2n}\varphi(\tau)f\left(\varepsilon \tau^{-2n}\right)d\tau$, $t>1$, then, using  inequality \eqref{ocen1234'}, we get
\[
+\infty=\lim_{t \rightarrow \infty}h(t)\geq  G(\gamma_{3}(1,x)) \geq h(t)\,.
\]
Clearly $h$ is a continuous function in $t \in (1,\infty)$, and the latter implies that there exists some $t^{*}$ satisfying $t\leq t^{*}\leq \infty$, such that 
\[
G(\gamma_{3}(1,x)) = h(t^{*})\,,
\]
or equivalently
$$G(\gamma_{3}(1,x))=\frac{1}{C\varepsilon} \int_{1}^{t^{*}}\tau^{2n}\varphi(\tau)f\left(\varepsilon \tau^{-2n}\right)d\tau\,.$$
Plugging the last expression into the inequality \eqref{ocen1234'} yields
\begin{equation}\label{b-c.fin'}
h(t^{*})-G(\gamma_{3}(t,x))\geq h(t)\,,\quad \forall t>1.
\end{equation}
Letting $t \rightarrow t^{*}$ in \eqref{b-c.fin'} we see that $-\lim\limits_{t \rightarrow t^{*}}G(\gamma_{3}(t,x))\geq 0$, which in turn implies that \begin{equation}\label{bclimit'}\lim_{t \rightarrow t^{*}}G(\gamma_{3}(t,x))= 0.\end{equation} Finally, the properties of $G$ and \eqref{bclimit'} imply that $\lim\limits_{t \rightarrow t^{*}}\gamma_{3}(t,x)= +\infty$, for some finite time $t^{*}$, and by \eqref{why,blow'} the solution cannot be global, and our claim has now been proved.
The proof of Theorem \ref{onhn} is complete.
\end{proof}

\end{document}